\documentclass[article,12pt]{amsart}

\usepackage{amsfonts}
\usepackage{amsmath}
\usepackage{graphicx}
\usepackage{color}
\usepackage{epsfig}
\usepackage{multirow}

\numberwithin{equation}{section}
\theoremstyle{plain}
\newtheorem{thm}{Theorem}[section]
\newtheorem{rem}{Remark}[section]

\newtheorem{cor}{Corollary}[section]
\newtheorem{lem}{Lemma}[section]

\newcommand{\dd}{\: \mathrm{d}}
\newcommand{\dE}{\mathbb{E}}

\newcommand{\cA}{\mathcal{A}}
\newcommand{\cF}{\mathcal{F}}
\newcommand{\cH}{\mathcal{H}}
\newcommand{\cL}{\mathcal{L}}
\newcommand{\cN}{\mathcal{N}}
\newcommand{\cR}{\mathcal{R}}
\newcommand{\cS}{\mathcal{S}}
\newcommand{\cT}{\mathcal{T}}
\newcommand{\cW}{\mathcal{W}}
\newcommand{\wh}{\widehat}

\newcommand{\estt}{\wh{\theta}_{T}}
\newcommand{\estr}{\wh{\rho}_{T}}
\newcommand{\estdw}{\wh{D}_{T}}
\newcommand{\estz}{\wh{Z}_{T}}
\newcommand{\whv}{\wh{V}}

\newcommand{\whvc}{\whv^{\: 2}}
\newcommand{\hsp}{\hspace{0.5cm}}
\newcommand{\cvgps}{\hspace{0.3cm} \text{a.s.}}
\newcommand{\limT}{\lim_{T \rightarrow \, \infty}}

\font\calcal=cmsy10 scaled\magstep1
\def\build#1_#2^#3{\mathrel{\mathop{\kern 0pt#1}\limits_{#2}^{#3}}}
\def\liml{\build{\longrightarrow}_{}^{{\mbox{\calcal L}}}}
\def\limp{\build{\longrightarrow}_{}^{{\mbox{\calcal P}}}}
\def\videbox{\mathbin{\vbox{\hrule\hbox{\vrule height1ex \kern.5em
\vrule height1ex}\hrule}}}

\email{Bernard.Bercu@math.u-bordeaux1.fr}
\email{Frederic.Proia@inria.fr}
\email{Nicolas.Savy@math.univ-toulouse.fr}

\begin{document}

\title[On Ornstein-Uhlenbeck driven by Ornstein-Uhlenbeck processes]
{On Ornstein-Uhlenbeck driven by Ornstein-Uhlenbeck processes
\vspace{2ex}}
\author{Bernard Bercu}
\address{Universit\'e Bordeaux 1, Institut de Math\'ematiques de Bordeaux,
UMR 5251, and INRIA Bordeaux Sud-Ouest, team ALEA, 351 Cours de la Lib\'eration, 33405 Talence cedex, France.}
\author{Fr\'ed\'eric Proia}
\address{\vspace{-0.8cm}}
\author{Nicolas Savy}
\address{Universit\'e Paul Sabatier, Institut de Math\'ematiques de Toulouse,
UMR C5583, 31062 Toulouse Cedex 09, France.}

\thanks{}

\begin{abstract}
We investigate the asymptotic behavior of the maximum likelihood estimators of the unknown parameters of
positive recurrent Ornstein-Uhlenbeck processes driven by Ornstein-Uhlenbeck processes.
\end{abstract}

\maketitle

\vspace{-0.5cm}


\section{INTRODUCTION AND MOTIVATION}


Since the seminal work of Ornstein and Uhlenbeck \cite{OU30}, a wide literature is available on Ornstein-Uhlenbeck processes
driven by Brownian or fractional Brownian motions \cite{K04}, \cite{LS01}. Many interesting papers are also available on Ornstein-Uhlenbeck processes
driven by L\'evy processes
\begin{equation}  
\label{OUL}
\dd X_{t} =  \theta X_{t} \dd t  +  \dd L_{t}
\end{equation}
where $\theta<0$ and $(L_{t})$ is a 
continuous-time stochastic process
starting from zero with stationary and independent increments. We refer the reader to 
Barndorff-Nielsen and Shephard \cite{BNS01} for the mathematical foundation on Ornstein-Uhlenbeck processes
driven by L\'evy processes, and also to \cite{BN11} for a recent extension 
to fractional L\'evy processes. Parametric estimation results for Ornstein-Uhlenbeck driven
by $\alpha$-stable L\'evy processes are established in \cite{HL07} whereas nonparametric estimation results
are given in \cite{JVDV05}. Two interesting applications related to money exchange rates and stock prices
may be found in \cite{BNS01} and \cite{ON08}, see also the references therein.
\ \vspace{1ex} \par
To the best of our knowledge, no results are available on Ornstein-Uhlenbeck driven by Ornstein-Uhlenbeck processes
defined, over the time interval $[0,T]$, by
\vspace{0.1cm}
\begin{equation}
\label{MOD}
\left\{
\begin{array}[c]{ccccl}
\dd X_{t} & = & \theta X_{t} \dd t & + & \dd V_{t} \vspace{1ex}\\
\dd V_{t} & = & \rho V_{t} \dd t & + & \dd W_{t} 
\end{array}
\right.
\end{equation}
where $\theta < 0$, $ \rho \leq 0$ and $(W_{t})$ is a standard Brownian motion. 
For the sake of simplicity and without loss of generality, we choose the initial values
$X_0 = 0$ and $V_0 = 0$. Our motivation for studying \eqref{MOD}
comes from two observations. On the one hand, the increments of Ornstein-Uhlenbeck processes are not independent
which means that the weighted maximum likelihood estimation approach of \cite{HL07} does not apply directly to our situation.
On the other hand, Ornstein-Uhlenbeck driven by Ornstein-Uhlenbeck processes are clearly related with
stochastic volatility models in financial mathematics \cite{S00}. Furthermore, \eqref{MOD} is the continuous-time 
version of the first-order stable autoregressive process driven by a first-order autoregressive process recently investigated
in \cite{BercuProia11}. 
\ \vspace{1ex} \par
The paper organizes as follows. Section \ref{SC-MLE} is devoted to the maximum likelihood estimation
for $\theta$ and $ \rho$. We also introduce the continuous-time Durbin-Watson statistic which will allow us 
to propose a serial correlation test for Ornstein-Uhlenbeck driven by Ornstein-Uhlenbeck processes.
In Section \ref{SC-MR}, we establish the almost sure convergence as well as the asymptotic normality of our estimates. 
One shall realize that there is a radically different behavior of the estimator of $\rho$ in the two situations where $\rho< 0$ and $\rho=0$.
Our analysis relies on technical tools postponed to Section \ref{SC-TOOLS}. Finally, in Section \ref{SC-DW}, we propose a statistical procedure
based on the continuous-time Durbin-Watson statistic, in order to test whether or not $\rho=0$.


\section{MAXIMUM LIKELIHOOD ESTIMATION}

\label{SC-MLE}

The maximum likelihood estimator of $\theta$ is given by
\begin{equation} 
\label{DEFHAT}
\widehat{\theta}_T = \frac{\int_0^T X_t \dd X_t}{\int_0^T X_t^2 \dd t} = \frac{X_T^2 - T}{2\int_0^T X_t^2 \dd t}.
\end{equation}
In the standard situation where $\rho=0$, it is well-known that $\widehat{\theta}_T$ converges to $\theta$
almost surely. Moreover, as $\theta <0$, the process $(X_T)$ is positive recurrent and we have the asymptotic normality
\begin{equation*}
\sqrt{T}\Bigl(\widehat{\theta}_T - \theta\Bigr) \liml \cN(0, -2 \theta).
\end{equation*}
We shall see in Section \ref{SC-MR} that the almost sure limiting value of  $\widehat{\theta}_T$ and its asymptotic variance
will change as soon as $\rho< 0$. The estimation of $\rho$ requires the evaluation of the 
residuals generated by the estimation 
of $\theta$ at stage $T$. For all $0 \leq t \leq T$, denote
\begin{equation}
\label{ESTRES}
\whv_{t} = X_{t} - \estt \Sigma_{t}
\end{equation}
where
\begin{equation}
\label{SIG}
\Sigma_{t} = \int_{0}^{t} X_{s} \dd s.
\end{equation}
By analogy with \eqref{DEFHAT} and on the basis of the residuals \eqref{ESTRES}, we estimate
$\rho$ by
\begin{equation} 
\label{DEFHAR}
\widehat{\rho}_T  = \frac{\whv_T^2 - T}{2\int_0^T \whv_t^2 \dd t}.
\end{equation}
Therefore, we are in the position to define the continuous-time version of the discrete-time Durbin-Watson statistic 
\cite{BercuProia11}, \cite{DurbinWatson50}, \cite{DurbinWatson51}, \cite{DurbinWatson71},
\begin{equation} 
\label{DEFDW}
\widehat{D}_T  = \frac{2 \int_0^T \whv_t^2 \dd t - \whv_T^2 + T}{\int_0^T \whv_t^2 \dd t},
\end{equation}
which clearly means that $\widehat{D}_T=2(1-\widehat{\rho}_T)$. In Section \ref{SC-MR}, we shall make use of $\widehat{D}_T$
to build a serial correlation statistical test for the Ornstein-Uhlenbeck driven noise, that is to test whether
or not $\rho = 0$.


\section{MAIN RESULTS}

\label{SC-MR}

The almost sure convergences of our estimates are as follows.

\begin{thm}
\label{T-ASCVG}
We have the almost sure convergences
\begin{equation}
\label{ASCVGTR}
\limT \estt = \theta^{*},
\hspace{1cm}
\limT \estr = \rho^{*} 
\hspace{0.5cm}
 \textnormal{\cvgps}
\end{equation}
where
\begin{equation}
\label{LIMTR}
\theta^{*} = \theta + \rho
\hspace{1cm}\text{and}\hspace{1cm}
\rho^{*} = \frac{\theta \rho (\theta + \rho)}{(\theta + \rho)^2 + \theta \rho}.
\end{equation}
\end{thm}


\begin{proof}

We immediately deduce from \eqref{MOD} that
\begin{equation}
\label{DECOX}
\int_{0}^{T} X_{t} \dd X_{t} = \theta S_{T} + \rho P_{T} + M_{T}^{X}
\end{equation}
where
\begin{equation}
\label{DEFMX}
S_{T} = \int_{0}^{T} X_{t}^{\, 2} \dd t, \hspace{1cm} P_{T} = \int_{0}^{T} X_{t} V_{t} \dd t, \hspace{1cm} M_{T}^{X} = \int_{0}^{T} X_{t} \dd W_{t}.
\end{equation}
We shall see in Corollary \ref{COR-LIMX2} below that
\begin{equation}
\label{CVGS}
\limT \frac{1}{T} S_T = -\frac{1}{2 (\theta + \rho)} \hspace{1cm}\textnormal{\cvgps}
\end{equation}
and in the proof of Corollary \ref{COR-LIMVC2} that
\begin{equation}
\label{CVGP}
\limT \frac{1}{T} P_T = -\frac{1}{2 (\theta + \rho)} \hspace{1cm}\textnormal{\cvgps}
\end{equation}
Moreover, if $(\cF_{t})$ stands for the natural filtration of the standard Brownian motion $(W_{t})$, then
$(M_{t}^{X})$ is a continuous-time $(\cF_{t})-$martingale with quadratic variation $S_{t}$. 
Hence, it follows from the strong law of large numbers for continuous-time martingales given e.g.
in \cite{F76} or \cite{L78}, that $M_{T}^{X} = o(T)$ a.s. Consequently,
we obtain from \eqref{DECOX} that
\begin{equation}
\label{CVGNUMT}
\lim_{T \to \infty} \frac{1}{T} \int_{0}^{T} X_{t} \dd X_{t} = -\frac{\theta}{2(\theta+\rho)} -\frac{\rho}{2(\theta+\rho)}
= -\frac{1}{2} \hspace{1cm}\textnormal{\cvgps}
\end{equation}
which leads, via \eqref{DEFHAT}, to the first convergence in \eqref{ASCVGTR}.
The second convergence in \eqref{ASCVGTR} is more difficult to handle.
We infer from \eqref{MOD} that
\begin{equation}
\label{DECOV}
\int_{0}^{T} V_{t} \dd V_{t} =  \rho L_{T} + M_{T}^{V}
\end{equation}
where
\begin{equation}
\label{DEFMV}
L_{T} = \int_{0}^{T} V_{t}^{\, 2} \dd t \hspace{1cm} \text{and} \hspace{1cm} M_{T}^{V} = \int_{0}^{T} V_{t} \dd W_{t}.
\end{equation}
On the one hand, if $\rho<0$, it is well-known, see e.g. \cite{F76} page 728, that
\begin{equation}
\label{CVGL}
\limT \frac{1}{T} L_T = -\frac{1}{2 \rho} \hspace{1cm}\textnormal{\cvgps}
\end{equation}
In addition, $(M_{t}^{V})$ is a continuous-time $(\cF_{t})-$martingale with quadratic variation $L_{t}$. 
Consequently, $M_{T}^{V} = o(T)$ a.s. and we find from \eqref{DECOV} that
\begin{equation}
\label{CVGVV}
\lim_{T \to \infty} \frac{1}{T} \int_{0}^{T} V_{t} \dd V_{t} = -\frac{1}{2} \hspace{1cm}\textnormal{\cvgps}
\end{equation}
However, we know from It\^o's formula that
\begin{equation*}
\frac{1}{T} \int_{0}^{T} X_{t} \dd X_{t} = \frac{1}{2} \left( \frac{X_{T}^{\, 2}}{T} - 1 \right)
\hspace{0.5cm} \text{and} \hspace{0.5cm}
\frac{1}{T} \int_{0}^{T} V_{t} \dd V_{t} = \frac{1}{2} \left( \frac{V_{T}^{\, 2}}{T} - 1 \right).
\end{equation*}
Then, we deduce from \eqref{CVGNUMT} and \eqref{CVGVV} that
\begin{equation}
\label{CVGXV2}
\lim_{T \to \infty} \frac{X_{T}^{\, 2}}{T} = 0 
\hspace{1cm} \text{and} \hspace{1cm}
\lim_{T \to \infty} \frac{V_{T}^{\, 2}}{T} = 0 
\hspace{0.5cm} \textnormal{\cvgps}
\end{equation}
 As $X_T=\theta \Sigma_T+ V_T$, it clearly follows from \eqref{ESTRES} and \eqref{CVGXV2} that
\begin{equation}
\label{CVGNUMRH}
\lim_{T \to \infty} \frac{1}{2} \bigg( \frac{\whvc_{T}}{T} - 1 \bigg) = -\frac{1}{2} 
\hspace{0.5cm} \textnormal{\cvgps}
\end{equation}
Hereafter, we have from \eqref{DEFHAR} the decomposition
\begin{equation}
\label{DECORH}
\widehat{\rho}_T  = \frac{T}{2\wh{L}_T}  \bigg( \frac{\whvc_{T}}{T} - 1 \bigg)
\end{equation}
where
\begin{equation*}
\wh{L}_{T} = \int_{0}^{T} \whvc_{t}\dd t.
\end{equation*}
We shall see in Corollary \ref{COR-LIMVC2} below that
\begin{equation}
\label{CVGDENRH}
\limT \frac{1}{T} \wh{L}_T = -\frac{1}{2  \rho^{*}} \hspace{1cm}\textnormal{\cvgps}
\end{equation}
Therefore, \eqref{DECORH} together with \eqref{CVGNUMRH} and \eqref{CVGDENRH} directly imply
\eqref{ASCVGTR}. On the other hand, if $\rho = 0$, it is clear from \eqref{MOD} that for all $t \geq 0$,
$V_t=W_t$. Hence, we have from \eqref{ESTRES} and It\^o's formula that
\begin{equation}
\label{DECONUMO}
\whvc_{T} - T = 2 M_T^W -2W_T \Sigma_T (\estt - \theta) + \Sigma_T^2(\estt - \theta)^2
\end{equation}
and
\begin{equation}
\label{DECODENO}
\wh{L}_{T} = L_{T}- 2 (\estt - \theta)\int_{0}^{T} W_t \Sigma_t\dd t + (\estt - \theta)^2 \int_{0}^{T} \Sigma_t^2\dd t
\end{equation}
where
\begin{equation*}
L_{T} = \int_{0}^{T} W_{t}^{\, 2} \dd t \hspace{1cm} \text{and} \hspace{1cm} M_{T}^{W} = \int_{0}^{T} W_{t} \dd W_{t}.
\end{equation*}
It is now necessary to investigate the a.s. asymptotic behavior of $L_T$. We deduce from the self-similarity of the 
Brownian motion $(W_t)$ that
\begin{equation}
\label{SELFS}
L_T= \int_{0}^{T} W_{t}^{\, 2} \dd t \overset{\cL}{=}  T \int_{0}^{T} W_{t/T}^{\, 2} \dd t \overset{\cL}{=} T^{\, 2} \int_{0}^{1} W_{s}^{\, 2} \dd s= T^2 L
\end{equation}
Consequently, it clearly follows from \eqref{SELFS} that for any power $0<a<2$,
\begin{equation}
\label{CVGLINF}
\limT \frac{1}{T^a} L_T = +\infty \hspace{1cm}\textnormal{\cvgps}
\end{equation}
As a matter of fact, since $L$ is almost surely positive, it is enough to show that
\begin{equation}
\label{LAPLACE}
\limT \dE\left[\exp\Bigl(-\frac{1}{T^a} L_T \Bigr)\right]=0.
\end{equation}
However, we have from standard Gaussian calculations, see e.g. \cite{LS01} page 232, that
\begin{equation*}
\dE\left[\exp\Bigl(-\frac{1}{T^a} L_T \Bigr)\right]=\dE\left[\exp\Bigl(-\frac{T^2}{T^a} L \Bigr)\right]
= \frac{1}{\sqrt{\cosh(v_T(a))}}
\end{equation*}
where $v_T(a)=\sqrt{2T^{2-a}}$ goes to infinity, which clearly leads to \eqref{LAPLACE}.
Furthermore, $(M_{t}^{W})$ is a continuous-time $(\cF_{t})-$martingale with quadratic variation $L_{t}$. 
We already saw that $L_{T}$ goes to infinity a.s. which implies that $M_{T}^{W} = o(L_{T})$ a.s. In addition, we obviously have
$\Sigma_T^2 \leq T S_T$. One can observe that convergence \eqref{CVGS} still holds when $\rho=0$, which ensures
that $\Sigma_T^2 \leq T^2$ a.s. Moreover, we deduce from the strong law of large numbers for continuous-time martingales that
\begin{equation*}
(\estt - \theta)^2=O \left( \frac{\log T}{T} \right)  \hspace{1cm}\textnormal{\cvgps}
\end{equation*}
which implies that 
$\Sigma_T^2(\estt - \theta)^2=O ( T\log T )=o(L_T)$ a.s.
By the same token, as $X_T^2=o(T)$ and $W_T^2=o(T \log T)$ a.s., we find that
$$W_T\Sigma_T (\estt - \theta)=o(L_T) \hspace{1cm}\textnormal{\cvgps} $$
Consequently, we obtain from \eqref{DECONUMO} that
\begin{equation}
\label{CVGNUMO}
\whvc_{T} - T = o(L_T) \hspace{1cm}\textnormal{\cvgps}
\end{equation}
It remains to study the a.s. asymptotic behavior of $\wh{L}_T$. One can easily see that
\begin{equation*}
\int_{0}^{T} \Sigma_t^2\dd t \leq \frac{2}{\theta^2} ( S_T + L_T).
\end{equation*}
However, it follows from \eqref{CVGS} and \eqref{CVGLINF} that $S_T=o(L_T)$ a.s.
which ensures that
\begin{equation}
\label{CVGDENO1}
(\estt - \theta)^2\int_{0}^{T} \Sigma_t^2\dd t=o(L_T) \hspace{1cm}\textnormal{\cvgps}
\end{equation}
Via the same arguments,
\begin{equation}
\label{CVGDENO2}
(\estt - \theta)\int_{0}^{T} W_t \Sigma_t\dd t =o(L_T) \hspace{1cm}\textnormal{\cvgps}
\end{equation}
Then, we find from \eqref{DECODENO}, \eqref{CVGDENO1} and \eqref{CVGDENO2} that
\begin{equation}
\label{CVGDENO}
 \wh{L}_{T} = L_{T}(1+o(1)) \hspace{1cm}\textnormal{\cvgps}
\end{equation}
Finally, the second convergence in \eqref{ASCVGTR} follows from
\eqref{CVGNUMO} and \eqref{CVGDENO} which achieves the proof of Theorem \ref{T-ASCVG}.
\end{proof}

Our second result deals with the asymptotic normality of our estimates

\begin{thm}
\label{T-AN}
If $\rho<0$, we have the joint asymptotic normality
\begin{equation}
\label{JOINTCLT}
\sqrt{T} \begin{pmatrix}
\estt - \theta^{*} \\
\estr - \rho^{*}
\end{pmatrix}
\liml \cN(0, \Gamma)
\end{equation}
where the asymptotic covariance matrix
\begin{equation}
\label{GAM}
\Gamma = \begin{pmatrix}
\sigma_{\theta}^2 & \ell \\
\ell & \sigma_{\rho}^2
\end{pmatrix}
\end{equation}
with $\sigma_{\theta}^2 = - 2\theta^{*}$, $\ell  = {\displaystyle  \frac{2 \rho^{*} \left( (\theta^{*})^2 - \theta \rho \right) }{(\theta^{*})^2 + \theta \rho}}$
and
\begin{equation*}
\sigma_{\rho}^2 = -\frac{2 \rho^{*} \left( (\theta^{*})^6 + \theta \rho \left( (\theta^{*})^4 - \theta \rho \left( 2 (\theta^{*})^2 - \theta \rho \right) \right) \right) }{\left( (\theta^{*})^2 + \theta \rho \right)^3}. 
\end{equation*}
\noindent
In particular, we have 
\begin{equation}
\label{CLTT}
\sqrt{T} \left( \estt - \theta^{*} \right) \liml \cN(0, \sigma_{\theta}^2),
\end{equation}
and
\begin{equation}
\label{CLTR}
\sqrt{T} \Big( \estr - \rho^{*} \Big) \liml \cN(0, \sigma_{\rho}^2).
\end{equation}
\end{thm}
\begin{proof} We obtain from \eqref{DEFHAT} the decompostion
\begin{equation}
\label{DECT}
\estt - \theta^{*} = \frac{M_{T}^{X}}{S_{T}} + \frac{R_{T}^{X}}{S_{T}}
\end{equation}
where
\begin{equation*}
R_{T}^{X} = \rho \int_{0}^{T} X_{t} (V_{t} - X_{t}) \dd t = - \theta \rho \int_{0}^{T} \Sigma_{t} \dd \Sigma_{t} = -\frac{\theta \rho}{2} \Sigma_{T}^{\, 2}.
\end{equation*}
We shall now establish a similar decomposition for $\estr - \rho^{*}$. It follows from \eqref{ESTRES} that for all $0 \leq t \leq T$,
\begin{eqnarray*}
\whv_{t} & = & X_{t} - \estt \Sigma_{t} = V_{t} - (\estt - \theta) \Sigma_{t} = V_{t} - (\estt - \theta^{*}) \Sigma_{t} - \rho \Sigma_{t}, \\
& = &  V_{t} - \frac{\rho}{\theta} (X_t - V_t) - \frac{1}{\theta}(\estt - \theta^{*}) (X_t - V_t) = \frac{\theta^{*}}{\theta}V_t - 
\frac{\rho}{\theta} X_t - \frac{1}{\theta}(\estt - \theta^{*}) (X_t - V_t),
\end{eqnarray*} 
which leads to
\begin{equation}
\label{DECLHAT}
\wh{L}_{T} = I_{T} + ( \wh{\theta}_{T} - \theta^{*} ) \left( J_{T} + ( \wh{\theta}_{T} - \theta^{*} ) K_{T} \right),
\end{equation}
where
\begin{eqnarray*}
I_{T} & = & \frac{1}{\theta^2} \Big( \rho^2 S_{T} + (\theta^{*})^2 L_{T} - 2 \theta^{*} \rho P_{T} \Big),\\
J_{T} & = & \frac{1}{\theta^2} \Big( 2 \rho S_{T} + 2 \theta^{*} L_{T} - 2(\theta + 2\rho) P_{T} \Big), \\
K_{T} & = & \frac{1}{\theta^2} \Big( S_{T} + L_{T} - 2 P_{T} \Big).
\end{eqnarray*}
Then, we deduce from \eqref{DEFHAR} and \eqref{DECLHAT} that
\begin{equation}
\label{DECR1}
\wh{L}_{T} \left( \estr - \rho^{*} \right) = \frac{I_{T}^V}{2} + \frac{1}{2}( \wh{\theta}_{T} - \theta^{*} ) 
\left( J_{T}^V +  ( \wh{\theta}_{T} - \theta^{*} ) K_{T}^V \right)
\end{equation}
in which
$I_{T}^V = \whvc_{T} - T - 2\rho^{*} I_{T}$, $J_{T}^V = - 2\rho^{*} J_{T}$, and $K_{T}^K = - 2\rho^{*} K_{T}$.
At this stage, in order to simplify the complicated expression \eqref{DECR1}, we make repeatedly use of It\^o's formula. For all $0 \leq t \leq T$, we have
\begin{eqnarray*}
L_{t} & = & \frac{1}{2 \rho} V_{t}^{\, 2} - \frac{1}{\rho} M_{t}^{V} - \frac{t}{2 \rho}, \\
P_{t} & = & \frac{1}{\theta^{*}} X_{t} V_{t} - \frac{1}{2 \theta^{*}} V_{t}^{\, 2} - \frac{1}{\theta^{*}} M_{t}^{X} - \frac{t}{2 \theta^{*}}, \\
S_{t} & = & \frac{1}{2 \theta} X_{t}^{\, 2} + \frac{\rho}{2 \theta^{*} \theta} V_{t}^{\, 2} - \frac{\rho}{\theta^{*} \theta} X_{t} V_{t} - \frac{1}{\theta^{*}} M_{t}^{X} - \frac{t}{2 \theta^{*}},
\end{eqnarray*}
where the continuous-time martingales $M_{t}^{X}$ and $M_{t}^{V}$ were previously defined in \eqref{DEFMX} and \eqref{DEFMV}.
Therefore, it follows from tedious but straightforward calculations that
\begin{equation}
\label{DECR2}
\wh{L}_{T} \left( \estr - \rho^{*} \right) = C_{X} M_{T}^{X} + C_{V} M_{T}^{V} + \frac{J_{T}^V}{2} 
( \wh{\theta}_{T} - \theta^{*} ) + R_{T}^{V}
\end{equation}
where
\begin{equation*}
C_{V} = \frac{ (\theta^{*})^2 \rho^{*}}{\theta^2 \rho} \hspace{1cm} \text{and} \hspace{1cm} C_{X} = -\frac{ \rho (2 \theta + \rho) \rho^{*}}{\theta^2 \theta^{*}}.
\end{equation*}
The remainder $R_{T}^{V}$ is similar to $R_{T}^{X}$ and they play a negligible role.
The combination of \eqref{DECT} and \eqref{DECR2} leads to the vectorial expression
\begin{equation}
\label{DECTR}
\sqrt{T} \begin{pmatrix}
\estt - \theta^{*} \\
\estr - \rho^{*} 
\end{pmatrix} = \frac{1}{\sqrt{T}} A_{T} Z_{T} + \sqrt{T} R_{T}
\end{equation}
where
\begin{equation*}
A_{T} = \begin{pmatrix}
S_{T}^{-1} T  &  0 \\
B_{T} \wh{L}_{T}^{-1}T &  C_{V} \wh{L}_{T}^{-1} T 
\end{pmatrix},
\hspace{1cm} 
R_{T} = \begin{pmatrix}
S_{T}^{-1} R_{T}^{X} \\
\wh{L}_{T}^{-1} D_{T}
\end{pmatrix}
\end{equation*}
with $B_T=C_{X} + J_{T}^V(2S_{T})^{-1}$ and $D_{T}=R_T^V+ J_{T}^V(2S_{T})^{-1}R_{T}^X$. The leading term in \eqref{DECTR} is 
the continuous-time vector $(\cF_{t})-$martingale $(Z_{t})$ with predictable quadratic variation $\langle Z \rangle_{t}$ given by
\begin{equation*}
Z_{t} = \begin{pmatrix}
M_{t}^{X} \\
M_{t}^{V}
\end{pmatrix}
\hspace{1cm} \text{and} \hspace{1cm}
\langle Z \rangle_{t} = \begin{pmatrix}
S_{t} & P_{t} \\
P_{t} & L_{t}
\end{pmatrix}.
\end{equation*}
We deduce from \eqref{CVGS}, \eqref{CVGP} and \eqref{CVGL} that
\begin{equation}
\label{CVGA}
\limT A_{T} = A \hspace{0.5cm} \cvgps
\end{equation}
where $A$ is the limiting matrix given by
\begin{equation*}
A= \begin{pmatrix}
-2 \theta^{*} & 0 \\
-2\rho^{*} (C_{X} -  2 (\theta \rho)^{-1} \theta^{*} \rho^{*} ) 
 & -2\rho^{*} C_{V} 
\end{pmatrix}.
\end{equation*}
By the same token, we immediately have from \eqref{CVGS}, \eqref{CVGP} and \eqref{CVGL} that
\begin{equation}
\label{CVGZZ}
\limT \frac{\langle Z \rangle_{T}}{T} = \Lambda =
-\frac{1}{2 \theta^{*}}\begin{pmatrix}
1 & 1 \\
1 & \theta^{*} \rho^{-1}
\end{pmatrix} 
\hspace{0.5cm}  \cvgps
\end{equation}
Furthermore, it clearly follows from Corollary \ref{COR-LIMSTAT} below that
\begin{equation}
\label{CVGRESP}
\frac{X_{T}^{\, 2}}{\sqrt{T}} \limp 0
\hspace{1cm} \text{and} \hspace{1cm}
\frac{V_{T}^{\, 2}}{\sqrt{T}} \limp 0.
\end{equation}
Finally, as $\Gamma = A \Lambda A^{\prime}$, the joint asymptotic normality \eqref{JOINTCLT} follows from the conjunction of
\eqref{DECTR}, \eqref{CVGA}, \eqref{CVGZZ}, \eqref{CVGRESP} together with Slutsky's lemma and the central limit theorem for 
continuous-time vector martingales given e.g. in \cite{F76}, which achieves the proof of Theorem \ref{T-AN}.
\end{proof}

\begin{thm}
\label{T-AN0}
If $\rho=0$, we have the convergence in distribution
\begin{equation}
\label{CVGRH0}
T \, \estr \liml \cW
\end{equation}
where the limiting distribution $\cW$ is given by
\begin{equation}
\label{DEFW}
\cW = \frac{ \int_{0}^{1} B_{s} \dd B_{s}}{\int_{0}^{1} B_{s}^{\, 2} \dd s}= \frac{B_1^2 - 1}{2\int_0^1 B_s^2 \dd s}
\end{equation}
and $(B_{t})$ is a standard Brownian motion.
\end{thm}
\begin{proof}
Via the same reasoning as in Section 2 of \cite{F79}, it follows from
the self-similarity of the Brownian motion $(W_t)$ that
\begin{eqnarray}
\label{CVGCOUP}
\left( \int_{0}^{T} W_{t}^{\, 2} \dd t, \, \frac{1}{2} \left( W_{T}^{\, 2} - T \right) \right) & \overset{\cL}{=} & \left( T \int_{0}^{T} W_{t/T}^{\, 2} \dd t, \, \frac{T}{2} \left( W_{1}^{\, 2} - 1 \right) \right), \nonumber \\
& = & \left( T^{\, 2} \int_{0}^{1} W_{s}^{\, 2} \dd s, \, \frac{T}{2} \left( W_{1}^{\, 2} - 1 \right) \right).
\end{eqnarray}
Moreover, we obtain from \eqref{DECLHAT} that
\begin{equation}
\label{DECLHAT0}
\wh{L}_{T} = \alpha_T S_T + \beta_T P_T + \gamma_T L_T
\end{equation}
where
\begin{eqnarray*}
\alpha_{T} & = & \frac{1}{\theta^2} \, ( \estt - \theta )^{\! 2}, \\
\beta_{T} & = & -\frac{2}{\theta} \, ( \estt - \theta ) - \frac{2}{\theta^2} \, ( \estt - \theta )^{\! 2}, \\
\gamma_{T} & = & 1 + \frac{2}{\theta} \, ( \estt - \theta ) + \frac{1}{\theta^2} \, ( \estt - \theta )^{\! 2}. 
\end{eqnarray*}
By Theorem \ref{T-ASCVG}, $\estt$ converges a.s. to $\theta$ which implies that $\alpha_{T}$, $\beta_{T}$, 
and $\gamma_{T}$ converge a.s. to $0,0$ and $1$. Hence, we deduce from \eqref{CVGS}, \eqref{CVGP} and \eqref{DECLHAT0}
that
\begin{equation}
\label{EQDENP}
 \wh{L}_{T} = L_{T}(1+o(1)) \hspace{1cm}\textnormal{\cvgps}
\end{equation}
Furthermore, one can observe that $\whvc_{T}/T$ shares the same asymptotic distribution as $W_{T}^{\, 2}/T$.
Finally, \eqref{CVGRH0} follows from \eqref{CVGCOUP} and \eqref{EQDENP} together with the continuous mapping theorem.
\end{proof}

\begin{rem}
The asymptotic behavior of $\estr$ when $\rho <0$ and $\rho=0$ is closely related to the results previously established
for the unstable discrete-time autoregressive process, see \cite{ChanWei88}, \cite{F79}, \cite{W58}.
According to Corollary 3.1.3 of \cite{ChanWei88}, we can express
\begin{equation*}
\cW = \frac{\cT^{\, 2} - 1}{2 \cS}
\end{equation*}
where $\cT$ and $\cS$ are given by the Karhunen-Loeve expansions
\begin{equation*}
\cT = \sqrt{2} \, \sum_{n=1}^{\infty} \gamma_{n} Z_{n}\hspace{1cm} \text{and} \hspace{1cm} \cS = \sum_{n=1}^{\infty} \gamma_{n}^{\, 2} Z_{n}^{\, 2}
\end{equation*}
with $\gamma_{n} = 2 (-1)^{n}/((2n - 1) \pi)$ and $(Z_{n})$ is a sequence of independent random variables with
$\cN(0,1)$ distribution.
\end{rem}

\begin{rem}
For all $0 \leq t \leq T$, the residuals $\whv_{t}$ given by \eqref{ESTRES} depend on $\estt$. It would have been more natural to make use
of the estimator of $\theta$ at stage $t$ instead of stage $T$, in order to produce a recursive estimate. In this situation, Theorem \ref{T-ASCVG} still holds but we have been unable to prove Theorem \ref{T-AN}.
\end{rem}


\section{SOME TECHNICAL TOOLS}

\label{SC-TOOLS}

First of all, most of our results rely on the following keystone lemma.
\begin{lem}
\label{LERGO}
The process $(X_{t})$ is geometrically ergodic.
\end{lem}
\begin{proof}
It follows from \eqref{MOD} that
\begin{equation}
\label{AR2}
\dd X_{t} = (\theta + \rho) X_{t} \dd t - \theta \rho \Sigma_{t} \dd t + \dd W_{t}
\end{equation}
where we recall that
\begin{equation*}
\Sigma_{t} = \int_{0}^{t} X_{s} \dd s.
\end{equation*}
Consequently, if
\begin{equation*}
\Phi_{t} =
\begin{pmatrix}
X_{t} \\
\Sigma_{t}
\end{pmatrix},
\end{equation*}
we clearly deduce from \eqref{AR2} that
\begin{equation*}
\dd \Phi_{t} = A \Phi_{t} \dd t + \dd B_{t}
\end{equation*}
where
\begin{equation*}
 A =
\begin{pmatrix}
\theta+\rho & -\theta \rho \\
1 & 0
\end{pmatrix} \hspace{1cm} \text{and} \hspace{1cm} B_{t} =
\begin{pmatrix}
W_{t} \\
0
\end{pmatrix}.
\end{equation*}

\medskip

\noindent The geometric ergodicity of $(\Phi_{t})$ only depends on the sign of $\lambda_{\text{max}}(A)$, \textit{i.e.} the largest eigenvalue of $A$, which has to be negative. An immediate calculation shows that
\begin{equation*}
\lambda_{\text{max}}(A) = \max(\theta,\, \rho)
\end{equation*}
which ensures that $\lambda_{\text{max}}(A) < 0$ as soon as $\rho<0$. Moreover, if $\rho=0$, $(X_{t})$ is an ergodic Ornstein-Uhlenbeck process since 
$\theta < 0$, which completes the proof of Lemma \ref{LERGO}.
\end{proof}

\begin{cor}
\label{COR-LIMX2}
We have the almost sure convergence
\begin{equation}
\label{CVGXX}
\limT \frac{1}{T} S_T = -\frac{1}{2 (\theta + \rho)} \hspace{0.5cm} \textnormal{\cvgps}
\end{equation}
\end{cor}
\begin{proof}
According to Lemma \ref{LERGO}, it is only necessary to establish the
asymptotic behavior of $\dE[X_{t}^{\, 2}]$. Denote 
$\alpha_t = \dE[X_{t}^{\, 2}]$, $\beta_t = \dE[\Sigma_{t}^{\, 2}]$ and $\gamma_t = \dE[X_{t} \Sigma_{t}]$. 
One obtains from It\^o's formula that
\begin{equation*}
\frac{\partial U_{t}}{\partial t} = C U_{t} + I
\end{equation*}
where
\begin{equation*}
U_{t} =
\begin{pmatrix}
\alpha_{t} \\
\beta_{t} \\
\gamma_{t}
\end{pmatrix}, \hsp C =
\begin{pmatrix}
2(\theta+\rho) & 0 & -2 \theta \rho \\
0 & 0 & 2 \\
1 & - \theta \rho & \theta + \rho
\end{pmatrix}, \hsp I =
\begin{pmatrix}
1 \\
0 \\
0
\end{pmatrix}.
\end{equation*}
It is not hard to see that $\lambda_{\text{max}}(C) = \max(\theta+\rho,\, 2\theta,\, 2\rho)$.
On the one hand, if $\rho< 0$, $\lambda_{\text{max}}(C) <0$ which implies that
\begin{equation*}
\lim_{t \rightarrow \, \infty} U_{t} = -C^{-1} I.
\end{equation*}
It means that
\begin{equation*}
\lim_{t \rightarrow \, \infty} \alpha_{t} = -\frac{1}{2(\theta + \rho)}, \hspace{0.5cm} \lim_{t \rightarrow \, \infty} \beta_{t} = -\frac{1}{2 \theta \rho (\theta + \rho)}, \hspace{0.5cm} \lim_{t \rightarrow \, \infty} \gamma_{t} = 0.
\end{equation*}
Hence, \eqref{CVGXX} follows from Lemma \ref{LERGO} together with the ergodic theorem.
On the other hand, if $\rho=0$, $(X_{t})$ is a positive recurrent Ornstein-Uhlenbeck process and convergence \eqref{CVGXX} is well-known.
\end{proof}

\begin{cor}
\label{COR-LIMVC2}
If $\rho < 0$, we have the almost sure convergence
\begin{equation*}
\limT \frac{1}{T} \wh{L}_T = -\frac{(\theta + \rho)^2 + \theta \rho}{2 \theta \rho (\theta + \rho)} \hspace{0.5cm}  \textnormal{\cvgps}
\end{equation*}
\end{cor}
\begin{proof}
If $\rho < 0$, $(V_{t})$ is a positive recurrent Ornstein-Uhlenbeck process and it is well-known that
\begin{equation*}
\limT \frac{1}{T}L_{T} = -\frac{1}{2 \rho} \hspace{0.5cm} \cvgps
\end{equation*}
In addition, as $X_t= \theta \Sigma_t + V_t$, 
\begin{equation*}
\int_{0}^{T} X_{t} \Sigma_{t} \dd t= \frac{1}{\theta} ( S_T - P_T ).
\end{equation*}
However, we already saw in the proof of Corollary \ref{COR-LIMX2} that
\begin{equation*}
\limT \frac{1}{T} \int_{0}^{T} X_{t} \Sigma_{t} \dd t = 0 \hspace{0.5cm} \cvgps
\end{equation*}
which leads, via \eqref{CVGXX}, to the almost sure convergence
\begin{equation*}
\limT \frac{P_{T}}{T} = -\frac{1}{2 (\theta + \rho)} \hspace{0.5cm} \cvgps
\end{equation*}
Consequently, we deduce from \eqref{ASCVGTR} together with \eqref{DECLHAT} that
\begin{equation*}
\lim_{T \to \infty} \frac{1}{T} \wh{L}_T  =  \lim_{T \to \infty} \frac{1}{T} I_T= -\frac{(\theta + \rho)^2 + \theta \rho}{2 \theta \rho (\theta+\rho)} 
\hspace{0.5cm} \cvgps
\end{equation*}
which achieves the proof of Corollary \ref{COR-LIMVC2}.
\end{proof}

\begin{cor}
\label{COR-LIMSTAT}
If $\rho < 0$, we have the asymptotic normalities
\begin{equation*}
X_{T} \liml \cN\left( 0, -\frac{1}{2 (\theta + \rho)} \right) \hspace{1cm}\text{and} \hspace{1cm} V_{T} \liml \cN\left( 0, -\frac{1}{2 \rho} \right).
\end{equation*}
The asymptotic normality of $X_{T}$ still holds in the particular case where $\rho = 0$.
\end{cor}
\begin{proof}
This asymptotic normality is a well-known result for the Ornstein-Uhlenbeck process $(V_{t})$ with $\rho < 0$. In addition, one can observe that
for all $t \geq 0$, $\dE[X_{t}] = 0$. The end of the proof is a direct consequence of the Gaussianity of $(X_{t})$ together with Lemma \ref{LERGO} 
and Corollary \ref{COR-LIMX2}.
\end{proof}

\bigskip


\section{A STATISTICAL TESTING PROCEDURE}

\label{SC-DW}

Our purpose is now to propose a statistical procedure in order to test
\begin{equation*}
\cH_0\,:\,`` \rho = 0" \hspace{1cm} \text{against} \hspace{1cm} \cH_1\,:\,`` \rho < 0".
\end{equation*}
We shall make use of the Durbin-Watson statistic given by \eqref{DEFDW}. Its asymptotic properties are as follows.

\begin{thm}
\label{T-DW}
We have the almost sure convergence
\begin{equation}
\label{ASCVGDW}
\limT \estdw = D^{*} \hspace{0.5cm}\textnormal{\cvgps}
\end{equation}
where $D^{*} = 2\left( 1 - \rho^{*} \right)$. In addition, if $\rho < 0$, we have the asymptotic normality
\begin{equation}
\label{CLTDW1}
\sqrt{T} \left( \estdw - D^{*} \right) \liml \cN(0, \sigma^2_{D})
\end{equation}
where 
$$\sigma^2_{D} = 4 \, \sigma^2_{\rho}= -\frac{8 \rho^{*} \left( (\theta^{*})^6 + \theta \rho \left( (\theta^{*})^4 - \theta \rho \left( 2 (\theta^{*})^2 - \theta \rho \right) \right) \right) }{\left( (\theta^{*})^2 + \theta \rho \right)^3}.  $$
while, if $\rho=0$,
\begin{equation}
\label{CLTW2}
T \left( \estdw - 2 \right) \liml -2 \cW
\end{equation}
with $\cW$ given by \eqref{DEFW}.
\end{thm}
\begin{proof}
The proof of Theorem \ref{T-DW} is a straightforward application of \eqref{ASCVGTR}, \eqref{CLTR} and \eqref{CVGRH0} since 
$\estdw = 2\left(1 - \estr \right)$. 
\end{proof}

\noindent From now on, let us define the test statistic
\begin{equation*}
\estz = T^{\, 2} \left( \estdw - 2 \right)^{2}.
\end{equation*}
It follows from Theorem \ref{T-DW} that under $\cH_0$,
\begin{equation*}
\estz \liml 4 \cW^{\, 2}
\end{equation*}
while, under $\cH_1$,
\begin{equation*}
\limT \estz = +\infty \hspace{0.5cm} \cvgps
\end{equation*}
From a practical point of view, for a significance level $\alpha$ where $0< \alpha < 1$, the acceptance and rejection regions are given by 
$\cA = [0, z_{\alpha}]$ and $\cR = ]z_{\alpha}, +\infty[$
where $z_{\alpha}$ stands for the $(1-\alpha)$-quantile of the distribution of $4 \cW^{\, 2}$. The null hypothesis $\cH_0$ will not be rejected if the empirical value
\begin{equation*}
\estz \leq z_{\alpha},
\end{equation*}
and will be rejected otherwise. Assume to conclude that $\cH_0$ is rejected, which means that we admit the existence of a serial correlation 
$\rho < 0$. Then, the best way to produce unbiased estimates is to study the process given by \eqref{AR2}. As a matter of fact, for all $ t \geq 0$,
\begin{equation*}
X_{t} = (\theta + \rho) \Sigma_{t} - \theta \rho \Pi_{t} + W_{t}
\end{equation*}
where
\begin{equation*}
\Sigma_{t} = \int_{0}^{t} X_{s} \dd s \hsp \hsp \text{and} \hsp \hsp \Pi_{t} = \int_{0}^{t} \Sigma_{s} \dd s.
\end{equation*}
The maximum likelihood estimator of the vector 
$$
\vartheta=\begin{pmatrix} \theta+\rho \\ -\theta\rho \end{pmatrix}
$$ is given by
\begin{equation*}
\wh{\vartheta}_T = \left( \int_{0}^{T} \Phi_{t} \Phi_{t}^{\, \prime} \dd t \right)^{\! -1} \int_{0}^{T} \Phi_{t} \dd X_{t}
\end{equation*}
where $\Phi_t=
\begin{pmatrix}
X_{t} \
\Sigma_{t}
\end{pmatrix}^{\prime}
$.
We can show the almost sure convergence
\begin{equation*}
\limT \wh{\vartheta}_T  = \vartheta \hsp \cvgps
\end{equation*}
as well as the asymptotic normality
\begin{equation*}
\sqrt{T} \left( \wh{\vartheta}_T  - \vartheta \right) \liml \cN(0, \Delta)
\end{equation*}
where the asymptotic covariance matrix is given by
\begin{equation*}
\Delta = \begin{pmatrix}
-2 \theta^{*} & 0 \\
0 & -2 \theta\rho \, \theta^{*}
\end{pmatrix}.
\end{equation*}
Accordingly, the maximum likelihood estimator $\wh{\vartheta}_T$ is strongly consistent and one can see that its components 
are asymptotically independent.

\bigskip

\bibliographystyle{acm}
\bibliography{OUDW}

\end{document}